\documentclass[11pt]{amsart}
\usepackage{fullpage}
\usepackage{amsmath, amssymb, amsthm, amscd}
\usepackage{color}
\usepackage{mathrsfs}
\usepackage{amsaddr}

\newcommand{\qqq}{\mathbb{Q}}
\newcommand{\ccc}{\mathbb{C}}
\newcommand{\xxx}{X^{(2)}}
\newcommand{\wer}{{(2)}}
\newcommand{\xij}{x_{ij}}
\newcommand{\cxij}{\mathbb{Q}[x_{ij}]}
\newcommand{\cij}{\mathbb{C}_{(ij)}}

\newtheorem{thm}{Theorem}[section]
\newtheorem{prop}[thm]{Proposition}
\newtheorem{lem}[thm]{Lemma}

\newtheorem*{thm:main}{Theorem \ref{mainintro}}

\theoremstyle{definition}
\newtheorem{defn}[thm]{Definition}

\newtheorem*{notn}{Notation}
\newtheorem{remark}[thm]{Remark}

\theoremstyle{remark}

\title[Geometry of the intersection ring]{Geometry of the intersection ring and vanishing relations in the cohomology of the moduli space of parabolic bundles on a curve}

\author{Elisheva Adina Gamse and Jonathan Weitsman }
\address{Department of Mathematics, Northeastern University}
\thanks{Partially supported by NSF grant DMS 12--11819}
\date{\today}

\begin{document}

\begin{abstract}
We study the ring generated by the Chern classes of tautological line bundles on the moduli space of parabolic bundles of arbitrary rank on a Riemann surface. We show the Poincar\'e duals to these Chern classes have simple geometric representatives. We use this construction to show that the ring generated by these Chern classes vanishes below the dimension of the moduli space, in analogy with the Newstead-Ramanan conjecture for stable bundles.
\end{abstract}
\maketitle
\section{Introduction}

Let $G$ be a compact Lie group with maximal torus $T$, let $\Sigma^g$ be a compact, connected, oriented $2$-manifold of genus $g$, and let $p \in \Sigma^g$. For $t \in T$, let 
\[
R_g(t) = \{(A_1, \ldots, A_g, B_1, \ldots, B_g) \in G^{2g} \vert \prod_{i=1}^g [A_i, B_i] \sim t \}
\]
(where $\sim$ denotes conjugacy in $G$). The fundamental group $\pi_1(\Sigma^g \setminus \{p\})$ can be presented by generators $a_1, \ldots, a_g, b_1, \ldots, b_g, c$ with the relation $\prod_{i=1}^g [a_i, b_i]=c$, where $c$ can be thought of as representing the boundary curve of a small disc containing $p$; we choose such a set of generators. Then
\begin{equation*}
R_g(t) = \{ \rho \in Hom(\pi_1(\Sigma^g \setminus \{p\}),SU(N)) \vert \rho(c) \sim t\},
\end{equation*}
$G$ acts on $R_g(t)$ by conjugation, and $S_g(t)=R_g(t)/G$ is the space of characters of the fundamental group of $\Sigma^g \setminus \{p\}$ in $G$ where the conjugacy class of the image of $c$ is fixed.

When $t=e$, we get the moduli space $\bar{S_g} = \text{Hom}(\pi_1(\Sigma^g),G)/G$ of flat connections on $\Sigma^g$. 

Take $t=\xi \in Z(G)$ to be in the centre of $G$. In particular, take $G=SU(N)$ and $\xi = e^{2\pi i k / N}I \in Z(G)$, where $(k,N)=1$. Then the space 
\[
R_g(\xi) = \{(A_1, \ldots,A_g, B_1, \ldots, B_g) \in G^{2g} \vert \prod_{i=1}^g [A_i, B_i] = \xi\}
\]
is the moduli space of flat connections on a principal $G$-bundle over $\Sigma^g$. If we equip $\Sigma^g$ with a conformal structure, then $S_g(\xi)$ acquires a K\"ahler structure as a moduli space of stable rank $N$ vector bundles, of degree $k$ and fixed determinant, over the corresponding Riemann surface. In this case, the following generalisations of the Newstead-Ramanan conjecture have been established: 

\begin{thm}[\cite{tw}] The Chern classes of $S_g(\xi)$ vanish above degree $N(N-1)(g-1)$. 
\end{thm}
\begin{thm}[\cite{ek}] \label{genthis} The ring generated by the Chern classes of the vector bundle associated to $R_g(\xi) \to S_g(\xi)$ via the standard representation of $SU(N)$ on $\ccc^N$ vanishes in dimension strictly greater than $2N(N-1)(g-1)$. 
\end{thm}
Let $a_2, \ldots, a_r$ be the Chern classes of the vector bundle associated to $R_g(\xi)$ via the standard representation of $SU(N)$ on $\ccc^N$. Earl shows in \cite{earl} that the Pontryagin ring of $S_g(\xi)$ is contained in the subring of $H^*(S_g(\xi))$ generated by the $a_i$, and so Theorem \ref{genthis} implies 
\begin{thm}[\cite{ek}]
The Pontryagin ring of $S_g(\xi)$ vanishes in dimension strictly greater than $2N(N-1)(g-1)$.
\end{thm}
These results are generalisations of a conjecture of Newstead published in \cite{newstead}. For some other references on this subject, see \cite{ab, bar,bmw, bkvanishing,donaldson, ek2, gieseker, salher,  jkkw, ljjwtoric,jeffrey, lisajeffrey,jk, lk, ljjw,  kiem, kiemli, kn, kirwan2, kirwan,alinaquot,mocounts, meinrenken,     nr, ramanan, seshadri, st,sikora,  thaddeus, witten, yoshida, zaig}. 

We now take $t$ to be a generic element of $T$: 
\begin{defn} \label{generict}
Let $t_1, \ldots, t_N \in (0,1)$ be such that $t_1+\cdots + t_N \in \mathbb{Z}$, but the sum of any proper nonempty subset of the $t_j$ is not an integer, and let $t=\text{Diag}(e^{2\pi i t_1}, \ldots, e^{2\pi i t_N} )\in T$.
\end{defn}  Consider 
\[
R_g(t) = \{(A_1, \ldots, A_g, B_1, \ldots, B_g) \in G^{2g} \vert \prod_{i=1}^g [A_i, B_i] \sim t\}.
\]
Again, $G$ acts on $R_g(t)$ by conjugation; in this case, $S_g(t)=R_g(t)/G$  is a moduli space of rank $N$ vector bundles over $\Sigma^g$ with parabolic structure at the marked point $p$.

Consider the torus bundle $V_g(t) \to S_g(t)$ given by 
\[
\begin{split}
V_g(t) &= \{(A_1, \ldots, A_g, B_1, \ldots, B_g) \in G^{2g} \vert \prod_{i=1}^g [A_i, B_i]=t\} \\&= \{ \rho \in Hom(\pi_1(\Sigma^g\setminus \{p\}),SU(N))\vert \rho(c)=t\};
\end{split}
\]
then $S_g(t)=R_g(t)/G=V_g(t)/T$. 
For $1 \leq i, j \leq N$ with $i \neq j$, let $L_{ij}\to S_g(t)$ be the line bundle associated to $V_g(t)$ by the representation 
\begin{equation} \label{cijaction}
\begin{split}
T \times \ccc &\to \ccc\\
(\text{Diag}(e^{\sqrt{-1}\theta_1}, \ldots, e^{\sqrt{-1}\theta_N}),z)& \mapsto e^{\sqrt{-1}(\theta_i - \theta_j)}z.
\end{split}
\end{equation}
We will denote this representation of $T$ by $\ccc_{(ij)}$, and its weight by $\chi_{ij}$. Note that since $L_{ij} \otimes L_{jk} = L_{ik}$ and $L_{ij} = L_{ji}^*$, for $i,j,k$ all distinct, we have $c_1(L_{ij})+c_1(L_{jk})=c_1(L_{ik})$ and $c_1(L_{ij})=-c_1(L_{ji})$. We are interested in the subring of $H^*(S_g(t);\qqq)$ generated by the $c_1(L_{ij})$. 
When $G=SU(2)$, $V_g(t)$ is a circle bundle, and
\begin{thm}[\cite{jw}] \label{j} $c_1(L_{12})^{2g}=0$.
\end{thm}
We will sketch here the proof of Theorem \ref{j} found in \cite{jw}, as the purpose of this paper is to extend this technique to arbitrary rank. The idea is to find explicit geometric cycles Poincar\'e dual to the Chern class $c_1(V_g(t))$.  For $1 \leq i \leq g$, consider the sections $s_{A_i}$ of $L_{12}\to S_g(t)$ induced by the equivariant maps 
\begin{align*}
V_g(t) &\to \ccc_{(12)} \\
(A_1, \ldots, A_g, B_1, \ldots, B_g) &\mapsto (A_i)_{12},
\end{align*}
where the subscript $12$ denotes the $(1,2)$ matrix entry. These sections $s_{A_i}$ vanish on the cycles 
\[
D(A_i) = \{(A_1, \ldots, A_g, B_1, \ldots, B_g) \in V_g(t) \vert [A_i, t]=1\}/T.
\]
Define sections $s_{B_i}$ and cycles $D(B_i)$ similarly, and consider the intersection
\[
D:=D(A_1) \cap \cdots \cap D(A_g) \cap D(B_1) \cap \cdots \cap D(B_g).
\]
This is the image in $S_g(t)$ of the set of elements $(A_1, \ldots, A_g, B_1, \ldots, B_g) \in V_g(t)$ with the matrices $A_1, \ldots, A_g, B_1, \ldots, B_g$ all diagonal, so $\prod_{i=1}^g[A_i.B_i]=1$. Since $t\neq 1$ by Definition \ref{generict}, $D=\emptyset$, and so $c_1(L_{12})^{2g}=0$.

The purpose of this paper is to generalise this result to $G=SU(N)$. Our theorem is 
\begin{thm}
 \label{mainintro}
For $1 \leq i, j \leq N$ and $i \neq j$, let $k_{ij}$ be nonnegative integers. Then the cohomology class 
\begin{equation*}
{\prod_{\substack{1\leq i,j \leq N \\ i \neq j}}}c_1(L_{ij})^{k_{ij}} \in H^*(S_g(t);\qqq)
\end{equation*} 
vanishes whenever 
\begin{equation*}
\sum_{\substack{1\leq i,j \leq N \\ i \neq j}}k_{ij} \geq N(N-1)g - N + 2.
\end{equation*} 
\end{thm}

\begin{remark} The dimension of $S_g(t)$ is given by
\begin{align*}
\dim (S_g(t)) &= 2g\dim SU(N)-\dim T-\dim SU(N)\\
&= 2g(N^2-1) -(N-1) -(N^2-1)\\
&= 2g(N^2-1) - N^2 -N+2.
\end{align*}
Theorem \ref{mainintro} above says that monomials in the $c_1(L_{ij})$ vanish in degree $r$ for $r \geq 2gN(N-1) -2N +4$, which is well below the dimension of $S_g(t)$.
\end{remark}
The proof uses the same technique as \cite{jw} but the combinatorics of the intersections is much more complicated; we illustrate it here for the case $G=SU(3)$. 

For $1\leq i,j \leq 3$ with $i \neq j$, consider the line bundles $L_{ij} \to S_g(t)$ associated to $V_g(t)$ as above. Their Chern classes $c_{ij}=c_1(L_{ij})$ satisfy 
\begin{equation} \label{relationsthree}
\begin{split}
c_{ij}+c_{jk}+c_{ki} =0\\
c_{ij} + c_{ji} = 0.
\end{split}
\end{equation}
For $1 \leq m \leq g$, we can as before find sections $s_{A_m}^{ij}$ of $L_{ij} \to S_g(t)$ that are zero on 
\[
D_{ij}(A_m) = \{(A_1, \ldots, A_g, B_1, \ldots, B_g) \in V_g(t) \vert (A_m)_{ij}=0\}/T,
\]
and similarly sections $s^{ij}_{B_m}$ with zero locus $D_{ij}(B_m)$. The intersection 
\[
D:=\bigcap_{m=1}^g D_{12}(A_m)\cap D_{12}(B_m) \cap D_{13}(A_m) \cap D_{13}(B_m)
\]
is the image in $S_g(t)$ of the set of elements of $V_g(t)$ where $(A_m)_{12}=(A_m)_{13} = 0$ and $(B_m)_{12}=(B_m)_{13}=0$ for all $m$. Hence $(\prod_{m=1}^g[A_m,B_m])_{11}=1$, so again if $t$ is generic as in Definition \ref{generict}, $D=\emptyset$. Hence the monomial $c_{12}^{2g}c_{13}^{2g}$ vanishes, and similarly, so do $c_{23}^{2g}c_{21}^{2g}$ and $c_{31}^{2g}c_{32}^{2g}$. But any monomial in the $c_{ij}$ of degree at least $6g-1$ may be written using the relations \eqref{relationsthree} as a sum of monomials containing at least one such factor, and thus 
\begin{prop}
Let $G=SU(3)$. Then any monomial in the Chern classes of the $L_{ij}$ of degree at least $6g-1$ vanishes.
\end{prop}
The generalisation of this argument to higher rank requires more careful attention to the combinatorics and algebra of possible monomials.

\section{Combinatorial Preliminaries} \label{prelims}

We fix an integer $g \geq 2$; in the geometric application this is the genus of the $2$-manifold $\Sigma^g$. 

\begin{notn} For any set $X$, we will denote by $\xxx$ the set of 2-element subsets of $X$. \end{notn}

\begin{lem} \label{lines}
Let $X$ and $Y$ be finite sets with $\vert X \vert = n \geq 3$ and $\vert Y \vert \geq n(n-1)g-n+2$, and let $f:Y\to \xxx$ be a function. Then there exists $z \in X$ such that $\vert f^{-1}((X \setminus \{z\} ) ^{(2)}) \vert \geq (n-1)(n-2)g-n+3$.
\end{lem}
\begin{proof}
For $x \in X$, let $Q_x$ denote the set $(X\setminus \{x\})^{(2)}$. Suppose there is no $z$ for which $\vert f^{-1}(Q_z)\vert \geq (n-1)(n-2)g-n+3$. So for each $x \in X$, $\vert f^{-1}(Q_x)\vert \leq (n-1)(n-2)g-n+2$, and thus \begin{equation*}
\sum_{x \in X}\vert f^{-1}(Q_x)\vert \leq n(n-1)(n-2)g-n(n-2).
\end{equation*} But each element $\{u,v\} \in \xxx$ is contained in exactly $n-2$ of the sets $Q_x$. So $\sum_{x \in X}\vert f^{-1}(Q_x)\vert = (n-2)\vert Y \vert$. Hence \begin{align*}
\vert Y \vert &\leq \frac{n(n-1)(n-2)g - n(n-2)}{n-2} \\ &= n(n-1)g-n \\ &< \vert Y \vert.
\end{align*} This is a contradiction, so such a $z$ must exist. 
\end{proof}

\begin{defn} 
Let $X$ be a finite set. A \emph{block} in $X\times X$ is a subset $B \subset X \times X$ of the form $V \times (X \setminus V)$, where $V \subset X$ is a proper nonempty subset. If $\vert V \vert = h$ we call $B$ an \emph{$h$-by-$(\vert X \vert-h)$ block}.
\end{defn}

\begin{defn} For a finite set $X$, let $\mathcal{B}[X]$ be the set of blocks in $X \times X$. \end{defn}

\begin{notn} If $n \in \mathbb{N}$, we will write $[n]$ for the set $\{1, \ldots, n\}$. \end{notn}

\begin{remark} \label{linalg}
Suppose $A \in SU(N)$ and $1 \leq h \leq N$, let $V \subset [N]$ be a subset with $\vert V \vert = h$, and consider the block $B = V \times ([N] \setminus V) \subset [N] \times [N]$. If $A_{ij}=0$ for all $(i,j)\in B$, then there is some ordering of basis elements for which $A$ is upper block diagonal. More precisely, if $\sigma$ is a permutation of $[N]$ such that $\sigma(V)=[h]$, then the matrix $(C_{ij})=(A_{\sigma^{-1}(i)\sigma^{-1}(j)})$ has the form
$$\left(\begin{array}{ccc|ccc}
A_{\sigma^{-1}(1)\sigma^{-1}(1)} & \cdots & A_{\sigma^{-1}(1)\sigma^{-1}(h)} & 0 & \cdots & 0 \\
\vdots& \ddots& \vdots& \vdots &\ddots & \vdots\\
A_{\sigma^{-1}(N-h)\sigma^{-1}(1)} & \cdots & A_{\sigma^{-1}(N-h)\sigma^{-1}(h)} & 0 & \cdots & 0 \\
\hline
\vert & & \vert & \vert & & \vert \\
u_1 & \cdots & u_h & v_1 & \cdots & v_{N-h}\\
\vert & & \vert & \vert & & \vert \\

\end{array}\right).$$
Since $A$, and hence $C$, are unitary, the vectors $v_i$ form a basis for $\ccc^{N-h}$, and $u_i\cdot v_j=0$ for all $i,j$. Hence each of the $u_i$ must be zero, so the matrix $C$ is block diagonal. Hence $A_{ji}=0$ for all $(i,j)\in B$. Thus the condition $A_{ij}=0$ for all $(i,j) \in B$ implies that $A$ is block diagonal, with blocks of size $h$ and $N-h$, up to reordering of basis elements.
\end{remark}

\section{Algebraic Preliminaries} \label{polyring}

Let $X$ be a finite set. 

\begin{defn} \label{cxij} Let $\cxij$ be the ring $\qqq[\{x_{ij}\vert 1\leq i, j\leq \vert X \vert, i \neq j\}]$, where we adjoin variables $\xij$ for all ordered pairs $(i,j)$ with $i,j \in X$ and $i \neq j$.
\end{defn}

We will need the following lemmas. 

\begin{lem} \label{somez} 
Suppose $\vert X \vert = n \geq 3$, and let $p \in \cxij$ be a monomial of degree at least $n(n-1)g-n+2$. Then there exists some $z \in X$ such that if we factor $p$ as $p = qr$, where $q \in \qqq[\{x_{iz}, x_{zi} \vert i \in X \setminus \{z\}\}]$ and $r \in \qqq[\{\xij \vert i, j \in X \setminus \{z\}, i \neq j\}]$ are monomials, then $r$ has degree at least $(n-1)(n-2)g-n+3$. 
\end{lem}

\begin{proof}
Write $p =\lambda \prod_{i,j} \xij^{d_{ij}}$. Let $Y_{ij}$ be disjoint sets with $\vert Y_{ij}\vert = d_{ij}$, for each pair $(i,j)$ with $i, j \in X$ and $i\neq j$. Let $Y = \cup Y_{ij}$. We have \begin{equation*}
\vert Y \vert = \sum_{i,j} d_{ij} \geq n(n-1)g-n+2.
\end{equation*} 
Consider the function 
\begin{align*}
f: Y &\to \xxx\end{align*} given by \begin{equation*}f\vert_{Y_{ij}}=\{i,j\}.\end{equation*} Then \begin{equation*} 
\vert f^{-1}(\{i,j\})\vert = d_{ij}+d_{ji}
\end{equation*} for all $\{i,j\} \in \xxx$. By Lemma \ref{lines}, there exists $z \in X$ with \begin{equation*}
\vert f^{-1}((X\setminus \{z\})^\wer)\vert \geq (n-1)(n-2)g-n+3.
\end{equation*} But \begin{align*}\vert f^{-1}((X\setminus \{z\})^\wer)\vert &= \sum_{\{i,j\} \in (X\setminus \{z\})^\wer}d_{ij}+d_{ji} \\ &= \deg(r), \end{align*} and so $\deg(r) \geq (n-1)(n-2)g-n+3$.
\end{proof}

\begin{lem} \label{partition} 
Let $z \in X$ and let $w, h \in \mathbb{N}$ with $w+h=\vert X \vert - 1$. Let $\eta \in \qqq[\{x_{zi} \vert i \in X \setminus \{z\}\}]$ be a monomial of degree at least $\vert X \vert(\vert X \vert-1)g -\vert X \vert+2-2gwh$. Given a partition $X\setminus \{z\} = \{e_1, \ldots, e_h\} \sqcup \{f_1, \ldots, f_w\}$, factor $\eta$ as $\eta = \eta_h \eta_w$, with $\eta_h \in \qqq[\{x_{ze_i} \vert 1 \leq i \leq h\}]$ and $\eta_w \in \qqq[\{ x_{zf_i}\vert 1 \leq i \leq w\}]$. Then either the degree of $\eta_h$ is at least $gh(h+1)-h+1$, or the degree of $\eta_w$ is at least $gw(w+1)-w+1$. 
\end{lem}

\begin{proof}
Suppose $\deg(\eta_h)<gh(h+1)-h+1$. Then
\begin{align*}
\deg(\eta_w) &= \deg(\eta) - \deg(\eta_h)\\
&\geq \vert X \vert(\vert X \vert -1)g - \vert X \vert + 2 -2gwh - gh(h+1)+h\\
&=(w+h+1)(w+h)g-(w+h+1)+2 -2gwh -gh(h+1) + h\\
&=g(w^2 + 2hw +h^2 + w + h -2wh-h^2 -h) -w-h+1 +h\\
&=gw(w+1)-w+1. \qedhere
\end{align*}
\end{proof}

\begin{defn} \label{ringr} Let $I \subset \cxij$ be the ideal generated by $x_{ij}+x_{ji}$ and $x_{ij} + x_{jk} +x_{ki}$, for all triples of distinct elements $i,j,k \in X$. Let $R = \cxij / I$ be the quotient of $\cxij$ by this ideal. 
\end{defn} 
Note that the quotient preserves the grading by degree. If $\zeta \in \cxij$ we will write $[\zeta]$ for its image in $R$. 

\begin{lem} \label{forcerestintoline} Let $\xi \in \cxij$ be a monomial, and let $z \in X$. Then there exists a homogeneous polynomial $\eta \in \qqq[\{x_{zj}\vert j \in X \setminus \{z\}\}]$ of the same degree as $\xi$ such that $[\xi] = [\eta] \in R$. 
\end{lem}

\begin{proof}
If $\xi = \lambda \prod_{i,j}\xij^{d_{ij}}$ where $\lambda \in \qqq$, let\begin{equation*}
\eta = \lambda\prod_{i,j \neq z}(-x_{zi}+x_{zj})^{d_{ij}}\prod_{i \neq z}(-1)^{d_{iz}}x_{zi}^{d_{iz}}\prod_{j\neq z}x_{zj}^{d_{zj}}. 
\end{equation*}Then $[\xi]=[\eta]$, and the degree of each term of $\eta$ is equal to the degree of $\xi$.
\end{proof}

\begin{prop} \label{induction}Let $X$ be a finite set with $\vert X \vert \geq 2$. Let $R = \cxij / I$ as before. Let $\zeta \in \cxij$ be a monomial of degree at least $\vert X \vert ( \vert X \vert - 1)g - \vert X \vert + 2$. Then for each block $B \subset X \times X$ we can find a monomial $\psi_B$ in $\cxij$ such that 
\begin{equation*}
[\zeta] = \left[\sum_{B \in \mathcal{B}[X]}\psi_B\prod_{(i,j) \in B}x_{ij}^{2g}\right].
\end{equation*}
\end{prop}
\begin{proof}
By induction on $\vert X \vert$.

If $\vert X \vert = 2$, take $X = \{1,2\}$. The monomials in $\cxij$ of degree at least $2g$ are of the form $\lambda x_{12}^ax_{21}^b$ where $a+b\geq 2g$, $\lambda \in \qqq$. The set $\{(1,2)\}$ is a 1-by-1 block in $X \times X$. Since $[x_{12}]=[-x_{21}]$ in $R$, the class $[\lambda x_{12}^ax_{21}^b]=[\lambda(-1)^bx_{12}^{a+b-2g}x_{12}^{2g}] \in R$ is of the desired form.

Now suppose $\vert X \vert = n \geq 3$. Let $\zeta \in \cxij$ be a monomial of degree $d \geq n(n-1)g-n+2$. By Lemma \ref{somez}, there exists $z \in X$ such that if we factor $\zeta$ as $\zeta = qr$, where $q \in \qqq[\{x_{iz},x_{zi}\vert i \in X\setminus \{z\}\}]$ and $r \in \qqq[\{x_{ij}\vert i, j \in X\setminus \{z\}, i \neq j\}]$ are monomials, then the degree of $r$ is at least $(n-1)(n-2)g-n+3$. 

By the inductive hypothesis, for each block $C \subset (X \setminus \{z\})\times (X \setminus \{z\})$ we can find a monomial $\theta_C \in \qqq[\{x_{ij}\vert i, j \in X\setminus \{z\}, i \neq j\}]$ such that 
\begin{equation*}
[r] = \left[\sum_{C \in \mathcal{B}[X\setminus \{z\}]}\theta_C \prod_{(i,j) \in C}x_{ij}^{2g}\right],
\end{equation*} and so 
\begin{equation*}
[\zeta] = \left[qr\right] = \left[q \sum_{C \in \mathcal{B}[X\setminus \{z\}]}\theta_C \prod_{(i,j) \in C}x_{ij}^{2g}\right].
\end{equation*}
It suffices to show that each nonzero monomial in the sum can be written as a sum of terms having the desired form. For each $C$ with $\theta_C \neq 0$, consider 
\begin{equation*}
\left[q\theta_C \prod_{(i,j)\in C}x_{ij}^{2g}\right].
\end{equation*}
This is a monomial of degree $d\geq n(n-1)g-n+2$. Suppose the block $C$ is given by 
\[
C = \{e_1, \ldots, e_h\} \times \{f_1, \ldots, f_w\},
\]
where $h, w \geq 1$ and $X \setminus \{z\}$ is the disjoint union $X \setminus \{z\} =\{e_1, \ldots, e_h\} \sqcup \{f_1, \ldots, f_w\}$ (so $w+h=n-1$). By Lemma \ref{forcerestintoline}, we can find a homogeneous polynomial $p_1 + \cdots + p_m$, where $p_1, \ldots, p_m$ are monomials in $\qqq[\{x_{zj}\vert j \in X \setminus \{z\}\}]$, such that 
\[
\left[q \theta_C \prod_{(i,j)\in C}x_{ij}^{2g}\right]=\left[(p_1 + \cdots + p_m)\prod_{(i,j)\in C}x_{ij}^{2g}\right].
\] 
Again, it suffices to show that each monomial in the sum can be written as a sum of terms having the desired form, so consider 
\[
\left[p\prod_{(i,j) \in C}x_{ij}^{2g}\right],
\]
where $p \in \{p_1, \ldots, p_m\}$. Note that
\[
\deg(p)= d - 2gwh \geq n(n-1)g-n+2-2gwh.
\]Factor $p$ as $p=p_hp_w$ where $p_h \in \qqq[\{x_{ze_i}\vert 1 \leq i \leq h\}]$ and $p_w \in \qqq[\{x_{zf_j}\vert 1 \leq j \leq w\}]$ are monomials. By Lemma \ref{partition}, either $\deg(p_h) \geq gh(h+1)-h+1$ or $\deg(p_w) \geq gw(w+1)-w+1$; without loss of generality we assume the former. 

By the inductive hypothesis, for each block $D \subset \{e_1, \ldots, e_h, z\}\times \{e_1, \ldots, e_h, z\}$ we can find a monomial $\phi_D \in \qqq[\{x_{ij} \vert i, j \in \{e_1, \ldots, e_h,z\}, i \neq j\}]$ such that 
\[
\left[p_h\right] = \left[\sum_{D \in \mathcal{B}[\{e_1, \ldots, e_h, z\}]}\phi_D \prod_{(i,j) \in D}x_{ij}^{2g}\right],
\]
and so 
\[
\left[p \prod_{(i,j)\in C}x_{ij}^{2g}\right]=\left[ p_w \sum_{D \in \mathcal{B}[\{e_1, \ldots, e_h, z\}]}\phi_D \prod_{(i,j) \in D}x_{ij}^{2g}\prod_{(i,j) \in C}x_{ij}^{2g}\right].
\]
For each $D$, consider the monomial
\[
\left[p_w \phi_D \prod_{(i,j) \in D}x_{ij}^{2g}\prod_{(i,j) \in C}x_{ij}^{2g}\right].
\]Observe that $C \cap D = \emptyset$, and so 
\[
\left[p_w \phi_D \prod_{(i,j) \in D}x_{ij}^{2g}\prod_{(i,j) \in C}x_{ij}^{2g}\right] = \left[p_w \phi_D\prod_{(i,j)\in C \cup D}x_{ij}^{2g}\right].
 \]We may assume 
\[
D = \{e_{\sigma(1)}, \ldots, e_{\sigma(d)}\}\times \{e_{\sigma(d+1)}, \ldots, e_{\sigma(h)}, z\},
\]
for some permutation $\sigma$ of $[h]$ and $1 \leq d \leq h$, since $x_{ij}-x_{ji}\in I$. But
\[
C \cup D =\{e_1, \ldots, e_h\} \times \{f_1, \ldots, f_w\} \cup \{e_{\sigma(1)}, \ldots, e_{\sigma(d)}\} \times \{e_{\sigma(d+1)}, \ldots, e_{\sigma(h)}, z\}
\]
contains
\[
E : = \{e_{\sigma(1)}, \ldots, e_{\sigma(d)}\} \times \{e_{\sigma(d+1)}, \ldots, e_{\sigma(h)}, f_1, \ldots, f_w, z\} \in \mathcal{B}[X].
\]
So
\[
\left[p_w \phi_D \prod_{(i,j) \in D}x_{ij}^{2g}\prod_{(i,j) \in C}x_{ij}^{2g}\right] = \left[ \psi_E \prod_{(i,j) \in E}x_{ij}^{2g}\right]
\] for some $\psi_E \in \cxij$. We have shown that $[\zeta]$ has a representative in $\cxij$ that is a sum of monomials of this form, i.e.
\begin{equation*}
[\zeta] = \left[\sum_{B \in \mathcal{B}[X]}\psi_B \prod_{(i,j) \in B} x_{ij}^{2g}\right],
\end{equation*}
for some monomials $\psi_B$ in $\cxij$.
\end{proof}

\section{Proof of the main theorem}

\begin{defn} \label{sections} Suppose $u \in \{a_1, \ldots, a_g, b_1, \ldots, b_g\}$ is one of the chosen generators of $\pi_1(\Sigma^g\setminus \{p\})$, and define maps 
 \begin{align*}f^{ij}_u : V_g(t) &\to \ccc_{(ij)}\\
 \rho &\mapsto (\rho(u))_{ij}
 \end{align*} for each pair $(i.j)$ with $1 \leq i, j \leq N$ and $i \neq j$. These maps $f_u^{ij}$ are $T$-equivariant since $T$ acts on the matrix entry $(\rho(u))_{ij}$ with weight $\chi_{ij}$. These maps then induce sections 
 \begin{align*}
 s_u^{ij}: S_g(t) & \to V_g(t) \times_T \cij \\
 [\rho] &\mapsto (\rho, f_u^{ij}(\rho))
 \end{align*}
 of the line bundles $L_{ij}$. 
 
Let $D_u^{ij}$ be the image in $S_g(t)$ of the subspace $\{ \rho \in V_g(t) \vert (\rho(u))_{ij}=0\}$. Then the section $s_u^{ij}$ is nonzero on the complement of $D_u^{ij}$. 
\end{defn} 

To prove Theorem \ref{mainintro}, we will show that intersections of certain sets of these subspaces $D_u^{ij}$ are empty, and conclude that the corresponding polynomials in the Chern classes $c_1(L_{ij})$ are zero. This is the same technique that was used in \cite{jw}.

\begin{lem} \label{nowherezero}
Let $M$ be a manifold. Let $\mathcal{L}_i \to M$, $i = 1, \ldots, m$, be complex line bundles with sections $s_i:M\to \mathcal{L}_i$. If these sections have no common zeros, i.e. $s_1^{-1}(0) \cap \cdots \cap s_m^{-1}(0) = \emptyset$, then $c_1(\mathcal{L}_1) \cdots c_1(\mathcal{L}_m) = 0$. 
\end{lem}

\begin{proof} 
Consider the vector bundle $E:=\mathcal{L}_1 \oplus \cdots \oplus \mathcal{L}_m \to M$ with the section \begin{equation*}\sigma :=(s_1, \cdots , s_m) : M \to E.
\end{equation*} The section $\sigma$ is nowhere zero, so the Euler class $e(E) =0$. Since $c_m(E) = e(E)$, the top Chern class $c_m(E)=c_1(\mathcal{L}_1)\cdots c_1(\mathcal{L}_m)$ is equal to zero.
\end{proof}

\begin{lem} \label{vanishing}
Let $0<h<N$, let $B \subset [N] \times [N]$ be an $h$-by-$N-h$ block, and let 
\begin{equation*}
\zeta = \prod_{(i,j)\in B}c_1(L_{ij})^{2g} \in H^{4gh(N-h)}(S_g(t)).
\end{equation*} Then $\zeta = 0$. 
\end{lem}

\begin{proof}
Consider the sections $s_u^{ij}:S_g(t)\to L_{ij}$ (as defined in \ref{sections}), for $(i,j)\in B$ and generators $u \in \{a_1, \ldots, a_g, b_1, \ldots, b_g\}$ of $\pi_1(\Sigma^g\setminus \{p\})$. We have $(s_u^{ij})^{-1}(0) = D_u^{ij}$. By Lemma \ref{nowherezero}, it suffices to show that \begin{equation*}
D:=\bigcap_{(i,j)\in B}(D_{a_1}^{ij} \cap \cdots \cap D_{a_g}^{ij} \cap D_{b_1}^{ij} \cap \cdots \cap D_{b_g}^{ij})=\emptyset.
\end{equation*} 
By definition, $D$ is the image in $S_g(t)$ of the set of homomorphisms $\rho \in V_g(t)$ such that the $(i,j)^{th}$ entry of $\rho(u)$ is zero for all $(i,j)\in B$ and all $u \in \{a_1, \ldots, a_g, b_1, \ldots, b_g\}$. Suppose $\rho \in D$. By Remark \ref{linalg}, we can find a permutation $\sigma$ of $[N]$ such that the matrices
\[
(\Phi(u))_{ij}):=(\rho(u))_{\sigma^{-1}(i)\sigma^{-1}(j)},
\]
for $u \in \{a_1, \ldots, a_g, b_1, \ldots, b_g\}$ are all block diagonal with blocks of size $h$ and $N-h$. So $\prod_{i=1}^g[\Phi(a_i),\Phi(b_i)]$ is also block diagonal with blocks of size $h$ and $N-h$, and each block has determinant equal to $1$. Let $E \in SU(N)$ be a product of elementary matrices representing this permutation $\sigma$ of basis elements, so that 
\[
\rho(u) = E^\dagger \Phi(u)E \text{ for all }u\in\{a_1,\ldots, a_g,b_1,\ldots, b_g\}.
\]
In particular, the matrix $EtE^\dagger$ is obtained from $t \in T$ by permuting the diagonal entries. Then 
\begin{align*}
t = \rho(c) &= \prod_{i=1}^g[\rho(a_i), \rho(b_i)] \\
&=\prod_{i=1}^g[E^\dagger \Phi(a_i)E, E^\dagger \Phi(b_i)E]\\
&=\prod_{i=1}^g E^\dagger [\Phi(a_i), \Phi(b_i)]E \\
&=E^\dagger\left( \prod_{i=1}^g [\Phi(a_i),\Phi(b_i)]\right)E.
\end{align*}
Thus $E\rho(c)E^\dagger = EtE^\dagger = \prod_{i=1}^g[\Phi(a_i),\Phi(b_i)]$ is a diagonal matrix where the first $h$ diagonal entries have product equal to $1$. But this is impossible because for $\rho \in V_g(t)$ we chose $t=\rho(c)$ such that no $h$ of its diagonal entries could have product equal to $1$ (see Definition \ref{generict}). Hence the set $D$ of such $\rho$ is empty. 
\end{proof}

\begin{proof}[Proof of Theorem \ref{mainintro}] Let $X = [N]$ and consider the rings $\cxij$ and $R$ as in Definitions \ref{cxij} and \ref{ringr}. Let $J\subset H^*(S_g(t))$ be the subring generated by the $c_1(L_{ij})$ for $1 \leq i,j \leq N$ and $i \neq j$. Since $c_1(L_{ij})=-c_1(L_{ji})$ and $c_1(L_{ij})+c_1(L_{jk})=c_1(L_{ik})$, the map 
\begin{align*}
\pi:R &\twoheadrightarrow J \\
[x_{ij}]&\mapsto c_1(L_{ij})
\end{align*} defines a ring homomorphism. Consider the element
\[
\prod_{\substack{1 \leq i, j \leq N \\ i \neq j}}c_1(L_{ij})^{k_{ij}} \in J.
\] It has a representative 
\[
\left[ \prod_{\substack{1 \leq i, j \leq N \\ i \neq j}}x_{ij}^{k_{ij}} \right]
\] in $R$. Suppose 
\[
\sum_{\substack{1 \leq i,j \leq N \\ i \neq j}}k_{ij} \geq N(N-1)g-N+2.
\]Then by Proposition \ref{induction}, for each block $B \subset [N] \times [N]$ we can find a monomial $\theta_B$ in $\cxij$ such that 
\[
\left[ \prod_{\substack{1 \leq i,j \leq N \\ i \neq j}}x_{ij}^{k_{ij}} \right] = \left[ \sum_{B \in \mathcal{B}[X]}\theta_B \prod_{(i,j) \in B}x_{ij}^{2g}\right].
\]So
\begin{align*}
\prod_{\substack{1 \leq i, j \leq N \\ i \neq j}}c_1(L_{ij})^{k_{ij}} &= \pi \left(\left[ \sum_{B \in \mathcal{B}[X]}\theta_B \prod_{(i,j) \in B}x_{ij}^{2g}\right] \right) \\
&= \sum_{B \in \mathcal{B}[X]}\pi(\theta_B)\prod_{(i,j) \in B}c_1(L_{ij})^{2g},
\end{align*} which vanishes by Lemma \ref{vanishing}.
\end{proof}

\bibliographystyle{plain} 

\bibliography{/home/adina/Dropbox/Maths/Math/refs}

\end{document}